\documentclass[11pt]{article}
\usepackage{mathptmx}
\usepackage{graphicx, amssymb, amsthm}
\usepackage[papersize={5.5in, 8.5in}, top=0.35in, bottom=0.5in, left=0.35in, right=0.35in]{geometry}

\usepackage[utf8]{inputenc}
\usepackage{pgfplots}
\pgfplotsset{compat=1.15}
\usepackage{mathrsfs}
\usetikzlibrary{arrows}
\usepackage{float}

\usepackage{amsmath}
\usepackage{color}
\usepackage{tikz}
\usepackage{hyperref}

\date{}

\theoremstyle{definition}
\newtheorem{df}{Definition} [section]

\theoremstyle{plain}

\newtheorem{lemma}[df]{Lemma}

\newtheorem{problem}[df]{Problem}

\newtheorem{condition}[df]{Condition}

\usepackage{algorithm}
\usepackage{algpseudocode}
\usepackage{tikz}
\usepackage{hyperref}

\definecolor{ffqqqq}{rgb}{1.,0.,0.}
\definecolor{qqwuqq}{rgb}{0.,0.39215686274509803,0.}
\definecolor{qqqqff}{rgb}{0.,0.,1.}
\definecolor{ffffqq}{rgb}{1.,1.,0.}
\definecolor{ffqqqq}{rgb}{1.,0.,0.}
\definecolor{ffffqq}{rgb}{1.,1.,0.}

\DeclareMathOperator{\arccosh}{arccosh}

\title{A $5$-chromatic same-distance graph in the hyperbolic plane}
\author{%
Geoffrey Exoo\\
Department of Mathematics and Computer Science\\
Indiana State University\\
Terre Haute, IN 47809,
\texttt{ge@cs.indstate.edu}
\and
Dan Ismailescu\\
Mathematics Department\\
Hofstra University\\
Hempstead, NY 11549,
\texttt{dan.p.ismailescu@hofstra.edu}
}

\begin{document}

\maketitle
\thispagestyle{empty}
\pagestyle{empty}

\begin{abstract}
The \emph{chromatic number of the plane} problem asks for the minimum number of colors so that each point of the plane can be assigned a single color with the property that no two points unit-distance apart are identically colored.
It is now  known that the answer is $5$, $6$, or $7$.
Here we consider the problem in the context of the hyperbolic plane.
We prove that there exists a distance $d\approx 1.375033509$ so that every
$4$-coloring of the hyperbolic plane contains two points distance $d$ apart,
which are identically colored.
\end{abstract}

\section{\bf Introduction}
$\phantom{assa}$In 1950, Edward Nelson raised the problem of determining the minimum number of colors that are needed to color the points of the Euclidean plane so that no two points unit distance apart are assigned the same color. This number is referred to as the {\it chromatic number of the plane}, and is denoted by $\chi(\mathbb{R}^2)$. The bounds $4\le \chi(\mathbb{R}^2)\le 7$ were proved shortly afterwards by Nelson and Isbell (cf. \cite {soifer}, p. 22--31).

There was no progress on the problem until 2018 when de Grey \cite{degrey}, and independently, the authors of this note \cite{exooismailescuchi5} were able to construct the first $5$-chromatic unit-distance graphs.
A flurry of activity ensued, pursuing two main goals. On one hand, progressively smaller 5-chromatic unit-distance graphs were found by Heule \cite{heule} and subsequently by Parts \cite{parts1}; the current record is a graph of order $509$.

On the other hand, since all these constructions rely on computer assistance it may be desirable to develop techniques that would make a traditional verification possible. A first contribution in this direction is due to Parts \cite{parts2}.
A well documented and entertaining history of the problem is presented by Soifer in his monograph \cite{soifer}. Several variations of the original problem have been recently considered. We direct the interested reader to the Polymath16 research threads \cite{polymath16} for the most current developments.

In 2012, Kahle \cite{kahle} raised the analogous question if one replaces the Euclidean plane by the hyperbolic plane. While the choice of distance for the Euclidean case is not important due to the existence of homotheties. In the hyperbolic setting, however, the chromatic number will depend on the choice of the distance $d$.

For a given $d>0$, let $\chi(\mathbb{H}^2,d)$ denote the minimum number of colors needed to color the hyperbolic plane, $\mathbb{H}^2$, so that no two points distance $d$ apart are identically colored. We refer to this quantity as the chromatic number of the hyperbolic plane with distance $d$.

There are only a few papers in the literature dedicated studying this quantity. Kloeckner \cite{kloeckner} proved that $\chi(\mathbb{H}^2,d)\le 12$ for all $d\le 2\ln(3/2)\approx 0.8109$, and that $\chi(\mathbb{H}^2,d)\le 4\lceil d/\ln{3}\rceil+4$ for all $d\ge 3\ln(3)\approx 3.2958$. These results were improved by Parlier and Petit \cite{PP} who proved the following upper bounds.
\begin{align}\label{parlierpetit}
\chi(\mathbb{H}^2,d)\le
\begin{cases}
9, \quad &\text{if} \,\, d\le 2\ln{2} \approx 1.3863,\\
12, \quad &\text{if} \,\, d\le 2\ln{3} \approx 2.1972,\\
5\lceil d/\ln{4}\rceil+5, \quad &\text{if} \,\, d\ge 2\ln{3}.
\end{cases}
\end{align}

Both papers mentioned above employ the general strategy of using a hyperbolic checkerboard coloring, a
method attributed to Sz\'ekely \cite{szekely}.
The only result involving lower bounds is due to DeCorte and Golubev \cite{DG}, who very recently proved that the \emph{measurable} chromatic number of the hyperbolic plane is $\ge 6$ for all sufficiently large $d$. (apparently, $d\ge 12$ is a sufficient condition).

Despite all this work, the following question of Kahle \cite{kahle} remained open until now.
\begin{problem}
Do we have $\chi(\mathbb{H}^2, d) \ge 5$ for at least one explicit $d$?
\end{problem}
The purpose of this paper is to identify one such specific $d$.
We will be using the Poincar\'e disk model of the hyperbolic plane. In this case, the hyperbolic
distance formula for $\mathbb{H}^2=\{(x,y)\,|\, x^2 +y^2<1\}$ can be expressed as
\begin{equation}\label{hyperbolicdistance}
d_{\mathbb{H}}((x_1,y_1), (x_2,y_2))=\arccosh\left(1+2\cdot\frac{(x_1-x_2)^2+(y_1-y_2)^2}{(1-x_1^2-y_1^2)(1-x_2^2-y_2^2)}\right).
\end{equation}
See e. g. \cite{beardon}, p. 40, or \cite{wiki} for a readily reachable reference.
Hence, the distance between any two points  $P_1=(x_1,y_1)$ and $P_2=(x_2,y_2)$ in $\mathbb{H}^2$ is uniquely determined by the quantity
\begin{equation}\label{f}
f(P_1,P_2)=2\cdot\frac{(x_1-x_2)^2+(y_1-y_2)^2}{(1-x_1^2-y_1^2)(1-x_2^2-y_2^2)}.
\end{equation}

\section{Moser Spindles in the Hyperbolic Plane}

It is well known that $\chi(\mathbb{H}^2,d)\ge 4$ for every $d>0$. The argument is similar to the one used by Nelson to prove $\chi(\mathbb{R}^2)\ge 4$ and it relies on the fact that a certain 4-chromatic graph, known as the Moser spindle, can be embedded as a distance $d$ graph in the hyperbolic plane.

We include our own argument here as it will be the basis of our construction. More precisely, for every $d>0$ we will find the explicit coordinates of a distance $d$ embedding of the Moser spindle in $\mathbb{H}^2$.

Let $0<R<1$ be some arbitrary number and consider the points $A_1=(0,0)$, $A_2=(R,0)$, and $A_3=(R\cos{\alpha}, R\sin{\alpha})$ where $0<\alpha<2\pi$. The condition that $A_1A_2A_3$ is an equilateral triangle translates to $f(A_1,A_2)=f(A_1,A_3)= f(A_2,A_3)$ which
by equation \eqref{f} is equivalent to the following equality.
\begin{equation}\label{condition1}
\cos{\alpha}=(1+R^2)/2.
\end{equation}
In particular, note that necessarily $0<\alpha<\pi/3$. So far we have that
\begin{equation}\label{firsttriangle}
f(A_1,A_2)=f(A_1,A_3)= f(A_2,A_3)=\frac{2R^2}{1-R^2}=\frac{2\cos{\alpha}-1}{1-\cos{\alpha}}.
\end{equation}
Next, compute the coordinates of the point $A_4\neq A_1$ for which triangle $A_2A_3A_4$ is equilateral. Again, the condition $f(A_2,A_3)=f(A_2,A_4)=f(A_3,A_4)$ leads to $A_4=\left(R(1+\cos{\alpha}/(2\cos{\alpha}), R\sin{\alpha}/(2\cos{\alpha})\right)$.

Let $A_1A_5A_7A_6$ be the image of the rhombus $A_1A_2A_4A_3$ after a counterclockwise rotation around point $A_1=(0,0)$ by an angle $\beta$. The coordinates of the new vertices are
\begin{align*}
A_5&=(R\cos{\beta}, R\sin{\beta}), A_6=(R\cos(\alpha+\beta), R\sin(\alpha+\beta)), \quad \text{and}\\ A_7&=\left(\frac{R(\cos{\beta}+\cos(\alpha+\beta))}{2\cos{\alpha}}, \frac{R(\sin{\beta}+\sin(\alpha+\beta))}{2\cos{\alpha}}\right).
\end{align*}

So far, we have that $f(A_i,A_j)=2\cos{\alpha-1}/(1-\cos{\alpha})$ for all pairs $\{i,j\}=\{1,2\}$, $\{1,3\}$, $\{2,3\}$, $\{2,4\}$, $\{3,4\}$, $\{1,5\}$, $\{1,6\}$, $\{5,6\}$,$\{5,7\}$, $\{6,7\}$.

To complete the spindle, we require $f(A_4, A_7)= (2\cos{\alpha-1})/(1-\cos{\alpha})$, which after some calculations becomes equivalent to the following equality.
\begin{equation}\label{betaalpha}
\cos{\beta}=1-\frac{1-\cos{\alpha}}{8\cos^2{\alpha}(1+\cos{\alpha})}.
\end{equation}
Note that since $\alpha \in (0,\pi/3)$ and the function $1-(1-x)/(8x^2(1+x))$ defined over $(1/2,1)$ is strictly increasing with range $(5/6,1)$, it follows that $\beta$ is uniquely determined by the choice of $\alpha$, and eventually by our initial choice of $0<R<1$. It follows that for all $d>0$, a distance $d$ Moser spindle can be embedded in the hyperbolic plane $\mathbb{H}^2$ by taking
\begin{equation*}
R=\tanh\left(\frac{d}{2}\right), \alpha=\arccos\left(\frac{1+R^2}{2}\right), \beta=\arccos\left(1-\frac{1-\cos{\alpha}}{8\cos^2{\alpha}(1+\cos{\alpha})}\right).
\end{equation*}
In the figure \ref{fig1} below we illustrate three Moser spindles embedded as distance $d$ graphs in the hyperbolic plane where $d=1, 2, 3$.
\begin{figure}[ht]
\centering
\begin{tikzpicture}[line cap=round,line join=round,>=triangle 45,x=6.0cm,y=6.0cm,scale=0.85]
\clip(-1,-1) rectangle (1,1);
\draw [line width=1.pt,dash pattern=on 3pt off 3pt] (0.,0.) circle (6.cm);
\draw [color=qqqqff](-0.6911931925903639,0.08377552499804987) node[anchor=north west] {$d=2$};
\draw [shift={(-0.5798478458372839,0.8630125557695418)},line width=2.pt,color=qqqqff]  plot[domain=2.230762726728423:3.3924561696033413,variable=\t]({-0.5*0.28462993081065147*cos(\t r)+-0.8660254037844387*0.28462993081065147*sin(\t r)},{0.8660254037844387*0.28462993081065147*cos(\t r)+-0.5*0.28462993081065147*sin(\t r)});
\draw [shift={(-0.8287975777154483,0.5775595137931719)},line width=2.pt,color=qqqqff]  plot[domain=2.8179514345574037:4.342362064860824,variable=\t]({-0.5*0.1431559131613423*cos(\t r)+-0.8660254037844387*0.1431559131613423*sin(\t r)},{0.8660254037844387*0.1431559131613423*cos(\t r)+-0.5*0.1431559131613423*sin(\t r)});
\draw [shift={(-1.0103909139611986,0.24520276654774714)},line width=2.pt,color=qqqqff]  plot[domain=3.767765351470587:4.929513579245311,variable=\t]({-0.5*0.28464024903182766*cos(\t r)+-0.8660254037844387*0.28464024903182766*sin(\t r)},{0.8660254037844387*0.28464024903182766*cos(\t r)+-0.5*0.28464024903182766*sin(\t r)});
\draw [shift={(-0.9622315120969085,0.5256938675956728)},line width=2.pt,color=qqqqff]  plot[domain=3.107921267629851:4.269616016899165,variable=\t]({-0.5*0.44972491647672796*cos(\t r)+-0.8660254037844387*0.44972491647672796*sin(\t r)},{0.8660254037844387*0.44972491647672796*cos(\t r)+-0.5*0.44972491647672796*sin(\t r)});
\draw [line width=2.pt,color=qqqqff] (0.,0.)-- (-0.38079707799999984,0.6595598864697688);
\draw [line width=2.pt,color=qqqqff] (0.,0.)-- (-0.5139820892691019,0.5620030927944871);
\draw [line width=2.pt,color=qqqqff] (0.,0.)-- (-0.7052043527054019,0.28759777260447905);
\draw [line width=2.pt,color=qqqqff] (0.,0.)-- (-0.750608136710205,0.12887142935668117);
\draw [shift={(-0.9622315120969085,0.5256938675956728)},line width=2.pt,color=qqqqff]  plot[domain=3.107921267629851:4.269616016899165,variable=\t]({-0.5*0.44972491647672796*cos(\t r)+-0.8660254037844387*0.44972491647672796*sin(\t r)},{0.8660254037844387*0.44972491647672796*cos(\t r)+-0.5*0.44972491647672796*sin(\t r)});
\draw [shift={(-0.8263495971590675,0.7207062502935798)},line width=2.pt,color=qqqqff]  plot[domain=2.8907346549812503:4.0524050111048195,variable=\t]({-0.5*0.4497322800288064*cos(\t r)+-0.8660254037844387*0.4497322800288064*sin(\t r)},{0.8660254037844387*0.4497322800288064*cos(\t r)+-0.5*0.4497322800288064*sin(\t r)});
\draw [shift={(-0.9335633775462284,0.45668336779567403)},line width=2.pt,color=qqqqff]  plot[domain=3.5514372113211023:4.714250276958633,variable=\t]({-0.5*0.2841439471947517*cos(\t r)+-0.8660254037844387*0.2841439471947517*sin(\t r)},{0.8660254037844387*0.2841439471947517*cos(\t r)+-0.5*0.2841439471947517*sin(\t r)});
\draw [shift={(-0.9335633775462284,0.45668336779567403)},line width=2.pt,color=qqqqff]  plot[domain=3.5514372113211023:4.714250276958633,variable=\t]({-0.5*0.2841439471947517*cos(\t r)+-0.8660254037844387*0.2841439471947517*sin(\t r)},{0.8660254037844387*0.2841439471947517*cos(\t r)+-0.5*0.2841439471947517*sin(\t r)});
\draw [shift={(-0.7516511127151841,0.7177192230048923)},line width=2.pt,color=qqqqff]  plot[domain=2.446026247166669:3.6087817123186796,variable=\t]({-0.5*0.28466201513993494*cos(\t r)+-0.8660254037844387*0.28466201513993494*sin(\t r)},{0.8660254037844387*0.28466201513993494*cos(\t r)+-0.5*0.28466201513993494*sin(\t r)});
\draw [line width=2.pt,color=qqwuqq] (0.,0.)-- (0.1396195379683894,-0.8943092499898977);
\draw [line width=2.pt,color=qqwuqq] (0.,0.)-- (-0.13872149866231143,-0.8944566180963275);
\draw [line width=2.pt,color=qqwuqq] (0.,0.)-- (0.2453513866444226,-0.8712492211686165);
\draw [line width=2.pt,color=qqwuqq] (0.,0.)-- (-0.13872149866231143,-0.8944566180963275);
\draw [line width=2.pt,color=qqwuqq] (0.,0.)-- (-0.24447770639241043,-0.8715085619643139);
\draw [shift={(0.06199872805236922,-1.0265869135913699)},line width=2.pt,color=qqwuqq]  plot[domain=2.547159884132635:4.403707478248279,variable=\t]({-0.2700963446858647*0.24030812741145485*cos(\t r)+0.9628333005184929*0.24030812741145485*sin(\t r)},{-0.9628333005184929*0.24030812741145485*cos(\t r)+-0.2700963446858647*0.24030812741145485*sin(\t r)});
\draw [shift={(-0.06096066306216405,-1.0266520146373947)},line width=2.pt,color=qqwuqq]  plot[domain=2.4275229196019255:4.28407051371757,variable=\t]({-0.2700963446858647*0.24030610562364002*cos(\t r)+0.9628333005184929*0.24030610562364002*sin(\t r)},{-0.9628333005184929*0.24030610562364002*cos(\t r)+-0.2700963446858647*0.24030610562364002*sin(\t r)});
\draw [shift={(0.18914152890918035,-0.9907752354641824)},line width=2.pt,color=qqwuqq]  plot[domain=2.9755012752175096:4.832108531375994,variable=\t]({-0.2700963446858647*0.13208336837013207*cos(\t r)+0.9628333005184929*0.13208336837013207*sin(\t r)},{-0.9628333005184929*0.13208336837013207*cos(\t r)+-0.2700963446858647*0.13208336837013207*sin(\t r)});
\draw [shift={(-0.18814131372765175,-0.9909749884678389)},line width=2.pt,color=qqwuqq]  plot[domain=1.999121866474211:3.855729122632695,variable=\t]({-0.2700963446858647*0.13208195221149638*cos(\t r)+0.9628333005184929*0.13208195221149638*sin(\t r)},{-0.9628333005184929*0.13208195221149638*cos(\t r)+-0.2700963446858647*0.13208195221149638*sin(\t r)});
\draw [shift={(-0.06853174284914762,-1.0063492193988923)},line width=2.pt,color=qqwuqq]  plot[domain=2.118776372474165:3.975334252431388,variable=\t]({-0.2700963446858647*0.132094924021733*cos(\t r)+0.9628333005184929*0.132094924021733*sin(\t r)},{-0.9628333005184929*0.132094924021733*cos(\t r)+-0.2700963446858647*0.132094924021733*sin(\t r)});
\draw [shift={(0.06954830490600944,-1.0062761126970574)},line width=2.pt,color=qqwuqq]  plot[domain=2.8558961454188165:4.71245402537604,variable=\t]({-0.2700963446858647*0.1320854119400258*cos(\t r)+0.9628333005184929*0.1320854119400258*sin(\t r)},{-0.9628333005184929*0.1320854119400258*cos(\t r)+-0.2700963446858647*0.1320854119400258*sin(\t r)});
\draw [color=qqwuqq](-0.4304885855442934,-0.4683473295638293) node[anchor=north west] {$d=3$};
\draw [shift={(1.31303,0.64956)},line width=2.pt,color=ffqqqq]  plot[domain=3.408397725259951:3.7935964123140824,variable=\t]({1.*1.0705064850340702*cos(\t r)+0.*1.0705064850340702*sin(\t r)},{0.*1.0705064850340702*cos(\t r)+1.*1.0705064850340702*sin(\t r)});
\draw [line width=2.pt,color=ffqqqq] (0.,0.)-- (0.46212,0.);
\draw [line width=2.pt,color=ffqqqq] (0.,0.)-- (0.2804,0.36732);
\draw [shift={(1.31303,-0.23256)},line width=2.pt,color=ffqqqq]  plot[domain=2.48960733411859:2.8748007190902016,variable=\t]({1.*0.8821197481067975*cos(\t r)+0.*0.8821197481067975*sin(\t r)},{0.*0.8821197481067975*cos(\t r)+1.*0.8821197481067975*sin(\t r)});
\draw [shift={(0.6118672174452774,1.1847997479387702)},line width=2.pt,color=ffqqqq]  plot[domain=4.327172107182222:4.7123694621216465,variable=\t]({1.*0.8821245119204828*cos(\t r)+0.*0.8821245119204828*sin(\t r)},{0.*0.8821245119204828*cos(\t r)+1.*0.8821245119204828*sin(\t r)});
\draw [shift={(0.86332,0.6765)},line width=2.pt,color=ffqqqq]  plot[domain=3.4922934014554285:4.120218708280811,variable=\t]({1.*0.4505233674305917*cos(\t r)+0.*0.4505233674305917*sin(\t r)},{0.*0.4505233674305917*cos(\t r)+1.*0.4505233674305917*sin(\t r)});
\draw [shift={(0.08822508302762735,1.3305682752744463)},line width=2.pt,color=ffqqqq]  plot[domain=4.737656398299413:5.122849783271025,variable=\t]({1.*0.8821178955785899*cos(\t r)+0.*0.8821178955785899*sin(\t r)},{0.*0.8821178955785899*cos(\t r)+1.*0.8821178955785899*sin(\t r)});
\draw [shift={(0.9447410777656234,1.1195800169329426)},line width=2.pt,color=ffqqqq]  plot[domain=3.818860705075532:4.204059392129663,variable=\t]({1.*1.0705026957929626*cos(\t r)+0.*1.0705026957929626*sin(\t r)},{0.*1.0705026957929626*cos(\t r)+1.*1.0705026957929626*sin(\t r)});
\draw [shift={(1.2967388681161274,0.3107485196470109)},line width=2.pt,color=ffqqqq]  plot[domain=2.900087655267968:3.2852850102073923,variable=\t]({1.*0.8821197481067974*cos(\t r)+0.*0.8821197481067974*sin(\t r)},{0.*0.8821197481067974*cos(\t r)+1.*0.8821197481067974*sin(\t r)});
\draw [line width=2.pt,color=ffqqqq] (0.,0.)-- (0.11051155293070097,0.4487319554512836);
\draw [line width=2.pt,color=ffqqqq] (0.,0.)-- (0.4237055275388525,0.18442971355926224);
\draw [color=ffqqqq](0.35812930063456894,-0.05) node[anchor=north west] {$d=1$};
\draw [shift={(5.025527486774983E-4,-1.0021835941287078)},line width=2.pt,color=qqwuqq]  plot[domain=2.342836302219476:4.48834294900387,variable=\t]({-0.2700963446858647*0.06616988514422549*cos(\t r)+0.9628333005184929*0.06616988514422549*sin(\t r)},{-0.9628333005184929*0.06616988514422549*cos(\t r)+-0.2700963446858647*0.06616988514422549*sin(\t r)});
\begin{scriptsize}
\draw [fill=ffffqq] (0.,0.) circle (3.0pt);
\draw [fill=ffffqq] (-0.38079707799999984,0.6595598864697688) circle (2.5pt);
\draw [fill=ffffqq] (-0.6873315159887783,0.5994571381359046) circle (2.5pt);
\draw [fill=ffffqq] (-0.7052043527054019,0.28759777260447905) circle (2.5pt);
\draw [fill=ffffqq] (0.,0.) circle (3.0pt);
\draw [fill=ffffqq] (-0.38079707799999984,0.6595598864697688) circle (2.5pt);
\draw [fill=ffffqq] (0.,0.) circle (3.0pt);
\draw [fill=ffffqq] (-0.5139820892691019,0.5620030927944871) circle (2.5pt);
\draw [fill=ffffqq] (0.,0.) circle (3.0pt);
\draw [fill=ffffqq] (-0.7052043527054019,0.28759777260447905) circle (2.5pt);
\draw [fill=ffffqq] (0.,0.) circle (3.0pt);
\draw [fill=ffffqq] (-0.750608136710205,0.12887142935668117) circle (2.5pt);
\draw [fill=ffffqq] (-0.8003628431864664,0.4372559808892423) circle (2.5pt);
\draw [fill=ffffqq] (0.2453513866444226,-0.8712492211686165) circle (2.5pt);
\draw [fill=ffffqq] (-0.05763940720088828,-0.9706060294154009) circle (2.5pt);
\draw [fill=ffffqq] (0.058618126796958686,-0.9705444766783311) circle (2.5pt);
\draw [fill=ffffqq] (0.1396195379683894,-0.8943092499898977) circle (2.5pt);
\draw [fill=ffffqq] (-0.13872149866231143,-0.8944566180963275) circle (2.5pt);
\draw [fill=ffffqq] (0.,0.) circle (3.0pt);
\draw [fill=ffffqq] (0.2453513866444226,-0.8712492211686165) circle (2.5pt);
\draw [fill=ffffqq] (0.,0.) circle (3.0pt);
\draw [fill=ffffqq] (0.1396195379683894,-0.8943092499898977) circle (2.5pt);
\draw [fill=ffffqq] (0.,0.) circle (3.0pt);
\draw [fill=ffffqq] (-0.13872149866231143,-0.8944566180963275) circle (2.5pt);
\draw [fill=ffffqq] (-0.24447770639241043,-0.8715085619643139) circle (2.5pt);
\draw [fill=ffffqq] (0.,0.) circle (3.0pt);
\draw [fill=ffffqq] (0.,0.) circle (3.0pt);
\draw [fill=ffffqq] (0.,0.) circle (3.0pt);
\draw [fill=ffffqq] (0.,0.) circle (3.0pt);
\draw [fill=ffffqq] (0.2453513866444226,-0.8712492211686165) circle (2.5pt);
\draw [fill=ffffqq] (0.,0.) circle (3.0pt);
\draw [fill=ffffqq] (-0.13872149866231143,-0.8944566180963275) circle (2.5pt);
\draw [fill=ffffqq] (0.,0.) circle (3.0pt);
\draw [fill=ffffqq] (-0.24447770639241043,-0.8715085619643139) circle (2.5pt);
\draw [fill=ffffqq] (0.46212,0.) circle (2.5pt);
\draw [fill=ffffqq] (0.2804,0.36732) circle (2.5pt);
\draw [fill=ffffqq] (0.61185,0.30268) circle (2.5pt);
\draw [fill=ffffqq] (0.440219,0.52172) circle (2.5pt);
\draw [fill=ffffqq] (0.4237055275388525,0.18442971355926224) circle (2.5pt);
\draw [fill=ffffqq] (0.11051155293070097,0.4487319554512836) circle (2.5pt);
\end{scriptsize}
\end{tikzpicture}
\caption{Moser spindles for hyperbolic distances $d=1, 2, 3$}
\label{fig1}
\end{figure}
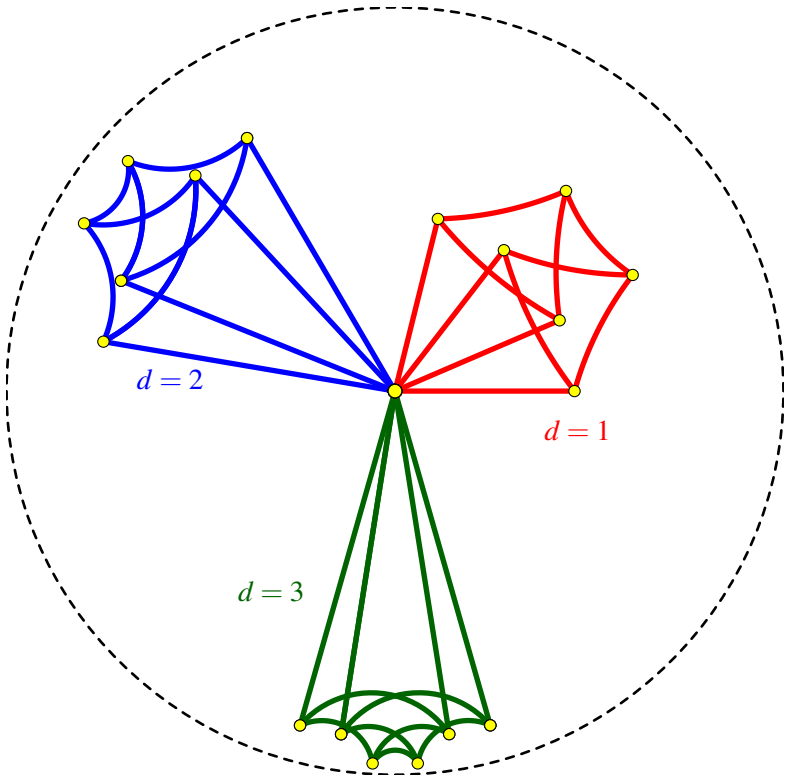

\section{A particularly promising distance}

As shown above, embedding the Moser spindle in the hyperbolic plane as a $d$-distance graph can be done for every $d>0$.
This leaves us with the freedom to choose a value of $d$ to our liking in our attempt to constructing a $5$-chromatic graph.
We settled on the following condition.

\begin{condition}\label{choice}
We choose $d>0$ such that for a distance $d$ embedding of the Moser spindle $\{A_1,A_2,A_3,A_4,A_5,A_6,A_7\}\in \mathbb{H}^2$, the circumcircle of $A_2A_4A_5$ and by symmetry, the circumcircle of $A_3A_6A_7$, both have circumradius equal to $d$.
\end{condition}

Recall that our embedding of the Moser spindle involved four parameters: $d, R, \alpha$ and $\beta$. Any one of these can be used to find the others. It turns out that condition \ref{choice} leads to the following value of $0<\alpha<\pi/3$.
\begin{equation}\label{cosalpha}
\cos{\alpha}=\frac{-1 + \sqrt{5} + \sqrt{22 - 2\sqrt{5}}}{8}\approx 0.6778371470.
\end{equation}
This allows us to find the exact values of the remaining parameters.
\begin{align*}
R&=\sqrt{2\cos{\alpha-1}}\approx 0.596384351,\\
\cos{\beta}&=1-\frac{1-\cos{\alpha}}{8\cos^2{\alpha}(1+\cos{\alpha})}\approx 0.9477621926,
\end{align*}
and for every edge $\{i,j\}$ of the Moser spindle
\begin{equation*}
f(A_i,A_j)=\frac{2\cos{\alpha}-1}{1-\cos{\alpha}} =\frac{5\sqrt{5}-13+(4+\sqrt{5})\sqrt{22-2\sqrt{5}}}{22},
\end{equation*}
Finally, the exact value of the hyperbolic distance $d$ announced in the abstract is
\begin{equation*}
d=\arccosh(1+ f(A_i,A_j))=\arccosh\left(\frac{\cos{\alpha}}{1-\cos{\alpha}}\right)\approx 1.375033509.
\end{equation*}

For the sake of brevity, we omit the computations. However, we will soon provide a simpler approach that allows verifying that these choices lead to condition \ref{choice} being satisfied. Given that none of the values of the parameters listed above are particularly simple, it is important we find a manageable way to handle such expressions.

Instead of using the explicit expression of $\cos{\alpha}$ given in \eqref{cosalpha}, we denote $\cos{\alpha}=c$ and note that
$c$ is the largest root of the equation $16c^4 + 8c^3 - 12c^2 - 2c + 1=0$.
Similarly, we denote $s=\sin{\alpha}$. All coordinates will be expressed in terms of $c, s$ and $R$ where
$s=\sqrt{1-c^2}, R=\sqrt{2c-1}$.

Finally, we denote $f=(2c-1)/(1-c)$, the common value of $f(A_i,A_j)$ where $\{i,j\}$ is an edge in our intended graph.

Since these relations and notations will be extensively used in the sequel we record them below for easy reference
\begin{align}\label{sr}
&16c^4 + 8c^3 - 12c^2 - 2c + 1=0, s=\sqrt{1-c^2}, R=\sqrt{2c-1},\notag \\
&f=\frac{2c-1}{1-c}, d=\arccosh(1+f)=\arccosh\left(\frac{c}{1-c}\right).
\end{align}

We are now in position to introduce our first $d$-distance graph.

\begin{lemma}\label{G9lemma}
Using the above notations let us consider the following points:
\begin{align*}
A_1&=(0,0), A_2=(R,0), A_3=(Rc, Rs), \\
A_4&=(R(-8c^3-4c^2+6c+3/2), Rs(-8c^3-4c^2+6c+1)),\\
A_5&=(R(-4c^3 + 4c^2 + 2c - 1), Rs(16c^3 - 4c^2 - 4c)),\\
A_6&=(R(4c^3 - 4c^2 + 1, Rs(-16c^3 + 12c^2 + 4c - 2),\\
A_7&=(R, Rs(16c^3 + 8c^2 - 8c - 2)),\\
A_8&=(R(-2c^2 + 2c + 1), Rs(8c^2 - 2c - 2)),\\
A_9&=(R(-8c^3 - 2c^2 + 4c + 3/2), Rs(8c^3 - 4c^2 + 1)).
\end{align*}
and the following edge set
\begin{align*}
E_1:=\{&\{1, 2\}, \{1, 3\}, \{1, 5\}, \{1, 6\}, \{2, 3\}, \{2, 4\}, \{2, 8\}, \{3, 4\}, \{3, 9\},\\
 &\{4, 7\},\{4, 8\}, \{5, 6\}, \{5, 7\}, \{5, 8\}, \{6, 7\}, \{6, 9\}, \{7, 9\}\}.
\end{align*}

Then for every $\{i,j\}\in E_1$ we have that $f(A_i,A_j)=f$. In other words, the graph with edge set $E_1$ is a $d$-distance graph
of order $9$ and size $17$.
\end{lemma}

\begin{proof}
Vertices $A_1$ through $A_7$ form a Moser spindle; their coordinates are the simplified versions of the coordinates introduced in the previous section subject to the conditions \eqref{sr}. The last two vertices are vertex $A_8$, adjacent to $A_2, A_4, A_5$, and vertex $A_9$, adjacent to $A_3, A_6, A_7$.

Note that the coordinates of each of vertices $A_i$ are of the form
\begin{equation}\label{coordinatesugly}
x_{A_i}=R(m_ic^3+n_ic^2+p_ic+q_i),\,\, \text{and}\,\, y_{A_i}=Rs(u_ic^3+v_ic^2+w_ic+z_i),
\end{equation}
where $m_i, n_i, p_i,q_i, u_i, v_i, w_i$, and $z_i$ are rational numbers.

Taking \eqref{sr} into account, for $i\neq j$ each of the following expressions
is a polynomial with rational coefficients in $c$.
\small{
\begin{align*}
&(x_{A_i}-x_{A_j})^2= (2c-1)\left((m_i-m_j)c^3+(n_i-n_j)c^2+(p_i-p_j)c+(q_i-q_j)\right)^2,\\
&(y_{A_i}-y_{A_j})^2=(2c-1)(1-c^2)\left((u_i-u_j)c^3+(v_i-v_j)c^2+(w_i-w_j)c+(z_i-z_j)\right)^2\\
&x_{A_i}^2+y_{A_i}^2=(2c-1)\left((m_ic^3+n_ic^2+p_ic+q_i)^2+(1-c^2)(u_ic^3+v_ic^2+w_ic+z_i)^2\right),\\
&x_{A_j}^2+y_{A_j}^2=(2c-1)\left((m_jc^3+n_jc^2+p_jc+q_j)^2+(1-c^2)(u_jc^3+v_jc^2+w_jc+z_j)^2\right).
\end{align*}}
Next, using \eqref{f}, it follows that $f(A_i,A_j)$ is a rational expression in $c$, and showing that
$f(A_i,A_j)=f=(2c-1)/(1-c)$ is simply a matter of verifying that a certain integer coefficients polynomial
is a multiple of the minimal polynomial of $c$ listed in \eqref{sr}, $16c^4 + 8c^3 - 12c^2 - 2c + 1$.

We performed these verifications in Maple.
In figure \ref{fig2} we present an embedding of this graph.
\begin{figure}[ht]
\centering
\begin{tikzpicture}[line cap=round,line join=round,>=triangle 45,x=1.0cm,y=1.0cm]
\clip(-1.025,-0.725) rectangle (10.375,8.968333333333309);
\draw [shift={(11.36577689,4.980373977)},line width=2.pt]  plot[domain=3.2227553370603923:3.8864091514031833,variable=\t]({1.*7.3474491634416275*cos(\t r)+0.*7.3474491634416275*sin(\t r)},{0.*7.3474491634416275*cos(\t r)+1.*7.3474491634416275*sin(\t r)});
\draw [shift={(11.36577689,-0.4394)},line width=2.pt]  plot[domain=2.396775937131926:3.0604301451510687,variable=\t]({1.*5.419774256983102*cos(\t r)+0.*5.419774256983102*sin(\t r)},{0.*5.419774256983102*cos(\t r)+1.*5.419774256983102*sin(\t r)});
\draw [shift={(7.3810935928992665,8.65409810703111)},line width=2.pt]  plot[domain=4.048734719469043:4.712388870989123,variable=\t]({1.*5.419774590993394*cos(\t r)+0.*5.419774590993394*sin(\t r)},{0.*5.419774590993394*cos(\t r)+1.*5.419774590993394*sin(\t r)});
\draw [shift={(4.223635749126399,10.554070405446959)},line width=2.pt]  plot[domain=4.375043252933229:5.038697460952371,variable=\t]({1.*5.419778088034865*cos(\t r)+0.*5.419778088034865*sin(\t r)},{0.*5.419778088034865*cos(\t r)+1.*5.419778088034865*sin(\t r)});
\draw [shift={(9.175662672108684,8.351485715627975)},line width=2.pt]  plot[domain=3.5490642466811138:4.212718061023905,variable=\t]({1.*7.347451742905268*cos(\t r)+0.*7.347451742905268*sin(\t r)},{0.*7.347451742905268*cos(\t r)+1.*7.347451742905268*sin(\t r)});
\draw [shift={(10.912963603016532,3.2177019806783687)},line width=2.pt]  plot[domain=2.7230845270951742:3.3867386786152545,variable=\t]({1.*5.419774207031144*cos(\t r)+0.*5.419774207031144*sin(\t r)},{0.*5.419774207031144*cos(\t r)+1.*5.419774207031144*sin(\t r)});
\draw [shift={(11.36577,-6.9080485)},line width=2.pt]  plot[domain=1.895446384027192:2.2344501616052765,variable=\t]({1.*8.769377565721866*cos(\t r)+0.*8.769377565721866*sin(\t r)},{0.*8.769377565721866*cos(\t r)+1.*8.769377565721866*sin(\t r)});
\draw [shift={(-1.6867466467182046,13.182909189300002)},line width=2.pt]  plot[domain=5.2010232364790205:5.540027014057106,variable=\t]({1.*8.769377503748611*cos(\t r)+0.*8.769377503748611*sin(\t r)},{0.*8.769377503748611*cos(\t r)+1.*8.769377503748611*sin(\t r)});
\draw [shift={(9.67425,3.42085)},line width=2.pt]  plot[domain=3.22275428265829:4.211059869120744,variable=\t]({1.*2.3007305389006802*cos(\t r)+0.*2.3007305389006802*sin(\t r)},{0.*2.3007305389006802*cos(\t r)+1.*2.3007305389006802*sin(\t r)});
\draw [shift={(7.063298577211756,7.439731261281167)},line width=2.pt]  plot[domain=3.2244135289635536:4.212719115426007,variable=\t]({1.*2.3007293194957583*cos(\t r)+0.*2.3007293194957583*sin(\t r)},{0.*2.3007293194957583*cos(\t r)+1.*2.3007293194957583*sin(\t r)});
\draw [shift={(8.70078971,5.642401581)},line width=2.pt]  plot[domain=3.2224840894963678:4.211060174657357,variable=\t]({1.*2.748841998091737*cos(\t r)+0.*2.748841998091737*sin(\t r)},{0.*2.748841998091737*cos(\t r)+1.*2.748841998091737*sin(\t r)});
\draw [line width=2.pt] (0.,0.)-- (5.96384,0.);
\draw [line width=2.pt] (0.,0.)-- (5.65523,1.9023336);
\draw [line width=2.pt] (0.,0.)-- (4.042514669,4.384689789);
\draw [line width=2.pt] (0.,0.)-- (2.4297781269113963,5.439769929842892);
\draw [shift={(0.,0.)},line width=1.2pt,dash pattern=on 3pt off 3pt]  plot[domain=0.:1.5707963267948966,variable=\t]({1.*10.*cos(\t r)+0.*10.*sin(\t r)},{0.*10.*cos(\t r)+1.*10.*sin(\t r)});
\draw (-0.785,0.34166666666666573) node[anchor=north west] {$A_1$};
\draw (6.0283333333333164,0.2) node[anchor=north west] {$A_2$};
\draw (4.148333333333318,4.92833333333332) node[anchor=north west] {$A_3$};
\draw (7.535,3.848333333333323) node[anchor=north west] {$A_4$};
\draw (6.095,6.0216666666666505) node[anchor=north west] {$A_7$};
\draw (4.375,8.155) node[anchor=north west] {$A_9$};
\draw (8.628333333333313,1.8616666666666617) node[anchor=north west] {$A_8$};
\draw (1.561666666666654,5.9816666666666505) node[anchor=north west] {$A_6$};
\draw (5.68166666666665,2.715) node[anchor=north west] {$A_5$};
\draw [shift={(8.56685,10.17262)},line width=2.pt]  plot[domain=4.373881983856806:4.7125815823217305,variable=\t]({1.*8.767848548100323*cos(\t r)+0.*8.767848548100323*sin(\t r)},{0.*8.767848548100323*cos(\t r)+1.*8.767848548100323*sin(\t r)});
\draw [shift={(12.782320193946338,3.6839990237244806)},line width=2.pt]  plot[domain=2.7228918157625666:3.061591414227492,variable=\t]({1.*8.769382262652336*cos(\t r)+0.*8.769382262652336*sin(\t r)},{0.*8.769382262652336*cos(\t r)+1.*8.769382262652336*sin(\t r)});
\begin{scriptsize}
\draw [fill=ffffqq] (0.,0.) circle (3.0pt);
\draw [fill=ffffqq] (5.96384,0.) circle (3.0pt);
\draw [fill=ffffqq] (7.381093,3.2343239) circle (3.0pt);
\draw [fill=ffffqq] (4.042514669,4.384689789) circle (3.0pt);
\draw [fill=ffffqq] (5.65523,1.9023336) circle (3.0pt);
\draw [fill=ffffqq] (2.4297781269113963,5.439769929842892) circle (3.0pt);
\draw [fill=ffffqq] (5.960936228902684,5.4202862222517) circle (3.0pt);
\draw [fill=ffqqqq] (8.568539,1.4032379) circle (3.0pt);
\draw [fill=ffqqqq] (4.770455441679783,7.249400608386462) circle (3.0pt);
\end{scriptsize}
\end{tikzpicture}
\caption{Embedding of a Moser spindle with two additional vertices}
\label{fig2}
\end{figure}
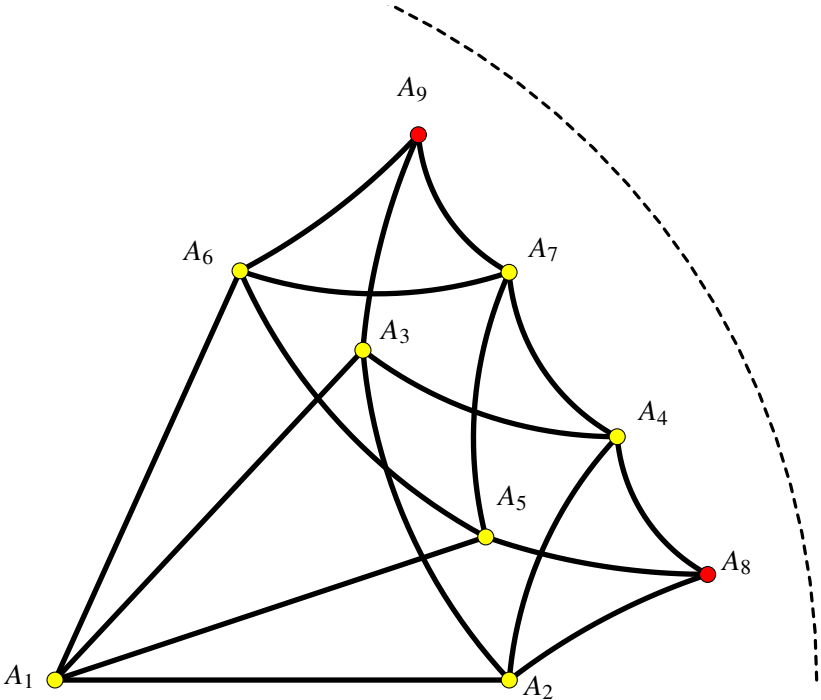

Rather than using the explicit coordinates of the vertex $A_i$ in the format \eqref{coordinatesugly}, we prefer to present
the location of $A_i$ by the octuple $[m_i, n_i, p_i, q_i, u_i, v_i, w_i,z_i]$.

\pagebreak

With this convention in place, the coordinates of the points $A_1$ through $A_9$ \\
\noindent introduced above can be written as:
\begin{align*}
&A_1=[0, 0, 0, 0, 0, 0, 0, 0], A_2=[0, 0, 0, 1, 0, 0, 0, 0], A_3=[0, 0, 1, 0, 0, 0, 0, 1],\\
&A_4=[-8, -4, 6, 3/2, -8, -4, 6, 1], A_5=[-4, 4, 2, -1, 16, -4, -4, 0],\\
&A_6=[4, -4, 0, 1, -16, 12, 4, -2], A_7= [0, 0, 0, 1, 16, 8, -8, -2], \\
&A_8=[0, -2, 2, 1, 0, 8, -2, -2], A_9=[-8, -2, 4, 3/2, 8, -4, 0, 1].
\end{align*}

\end{proof}

\section{Augmenting the graph}

Let $G_9$ be the graph of order $9$ and size $17$ defined in Lemma \ref{G9lemma}. Certainly, $G_9$ has chromatic number equal to $4$. The idea is to augment this graph by considering additional points. In the first phase, for every $1\le i<j\le 9$ we look for points $M$ for which $f(M, A_i)=f(M, A_j)=f$. This leads to a list of $19$ additional points, labeled $A_{10}$ through $A_{28}$.
Table \ref{table1} below contains to coordinates of these points and the neighbors of each such vertex among the vertices of $G_9$.

\begin{table}[hbtp]
\begin{centering}
\begin{tabular}{c|c|c}
   New Vertex        & Coordinates           & Neighbors in $G_9$       \\
    \hline
$A_{10}$ & $[0, 0, 1, 0, 0, 0, 0, -1]$ & $A_1$, $A_2$\\
$A_{11}$ & $[0, 2, 0, -1, 0, 0, 2, 0]$ & $A_1$, $A_3$\\
$A_{12}$ & $[16, 4, -6, -1, -16, 8, 0, 0]/2$ & $A_1$, $A_5$\\
$A_{13}$ & $[-16, 4, 2, 1, 48, -24, -8, 4]/2$ & $A_1$, $A_6$\\
$A_{14}$ & $[-96, 64, 36, -1, 592, -8, -164, -18]/58$ & $A_2$, $A_7$\\
$A_{15}$ & $[64, 112, 34, -9, 240, 72, -148, -12]/58$ & $A_2$, $A_8$\\
$A_{16}$ & $[8, 12, -4, -3, -16, 0, 8, 2]/2$ & $A_3$, $A_4$\\
$A_{17}$ & $[16, 4, -8, 1, -16, 8, 8, -2]/2$ & $A_3$, $A_5$\\
$A_{18}$ & $[16, 8, -10, -1, -16, -8, 12, 4]/2$ & $A_3$, $A_9$\\
$A_{19}$ & $[-104, -8, 126, 11, 16, 144, 52, -24]/58$ & $A_3$, $A_9$\\
$A_{20}$ & $[-100, 28, 52, 5, 176, -156, -8, 26]/29$ & $A_4$, $A_6$\\
$A_{21}$ & $[12, 10, -7, -3, 0, -4, 2, 1]$ & $A_4$, $A_7$\\
$A_{22}$ & $[-8, 0, 6, 1, 16, 0, -4, 0]/2$ & $A_4$, $A_7$\\
$A_{23}$ & $[32, 28, -16, -9, 16, 8, -8, -2]/2$ & $A_4$, $A_8$\\
$A_{24}$ & $[64, -12, -16, -1, -144, 88, 24, -10]/2$ & $A_5$, $A_7$\\
$A_{25}$ & $[12, 8, -6, -2, 0, 4, 0, -2]$ & $A_5$, $A_8$\\
$A_{26}$ & $[80, 24, 28, 25, 48, 200, -76, -14/58$ & $A_5$, $A_8$\\
$A_{27}$ & $[-96, 6, 65, -1, 128, -8, 10, 11]/29$ & $A_6$, $A_9$\\
$A_{28}$ & $[12, 8, -5, -3, -16, -4, 8, 3]$ & $A_7$, $A_9$
\end{tabular}
\caption{\small{Points in $\mathbb{H}^2$ with at least two neighbors among $A_1,A_2,\ldots, A_9$}}
\label{table1}
\end{centering}
\end{table}

Moreover, there are exactly six pairs $(i,j)$, $10\le i< j\le 28$ for which $f(A_i,A_j)=f$.
These accidental edges are $\{11, 18\}$, $\{11, 21\}$, $\{12, 21\}$, $\{12, 25\}$, $\{17, 19\}$, and  $\{17, 26\}$.
In the end, we obtain a graph $G_{28}$ of order $28$ and size $61$ whose distance $d$ embedding is shown if figure \ref{fig28}.
Vertices $A_1$ through $A_{9}$ are shown in yellow, while the new vertices, $A_{10}$ through $A_{28}$, appear in red.
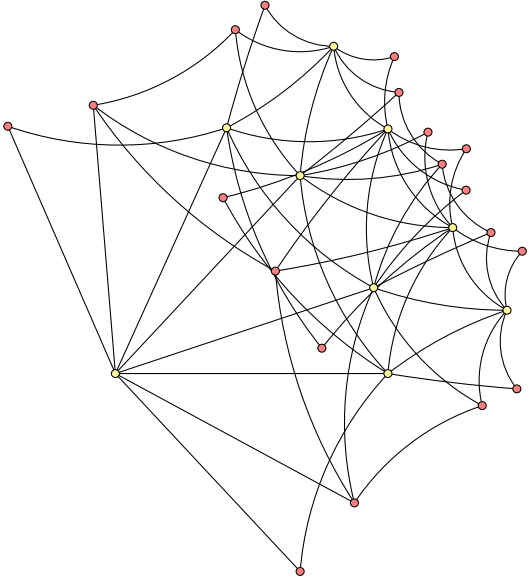
\begin{figure}[ht]
\centering
\begin{tikzpicture}[scale=6.0,line cap=round,line join=round,>=triangle 45]
\tikzstyle{gnode}=[draw=black, fill=yellow!50!white, shape=circle, minimum size=3pt, inner sep=0pt, outer sep=0pt];
\tikzstyle{rnode}=[draw=black, fill=red!50!white, shape=circle, minimum size=3pt, inner sep=0pt, outer sep=0pt];

\draw (0.0000,0.0000) -- (0.5964,0.0000);
\draw (0.0000,0.0000) -- (0.4043,0.4385);
\draw (0.0000,0.0000) -- (0.5652,0.1902);
\draw (0.0000,0.0000) -- (0.2433,0.5445);
\draw (0.0000,0.0000) -- (0.4043,-0.4385);
\draw (0.0000,0.0000) -- (-0.0484,0.5944);
\draw (0.0000,0.0000) -- (0.5230,-0.2866);
\draw (0.0000,0.0000) -- (-0.2354,0.5479);
\draw (0.5964,0.0000) arc(222.6749:184.6503:0.7347);
\draw (0.5964,0.0000) arc(175.3497:137.3251:0.5420);
\draw (0.5964,0.0000) arc(128.0246:108.6011:0.8769);
\draw (0.5964,0.0000) arc(137.3251:175.3497:0.7347);
\draw (0.5964,0.0000) arc(237.9317:207.6793:1.0174);
\draw (0.5964,0.0000) arc(260.6994:265.5721:3.3425);
\draw (0.4043,0.4385) arc(231.9754:270.0000:0.5420);
\draw (0.4043,0.4385) arc(175.3497:155.9263:0.8769);
\draw (0.4043,0.4385) arc(270.0000:231.9754:0.7347);
\draw (0.4043,0.4385) arc(279.3006:298.7240:0.8769);
\draw (0.4043,0.4385) arc(260.6994:288.6011:0.6469);
\draw (0.4043,0.4385) arc(222.6749:184.6503:0.5420);
\draw (0.4043,0.4385) arc(308.0246:312.8972:3.3425);
\draw (0.7381,0.3234) arc(241.2760:184.6503:0.2746);
\draw (0.7381,0.3234) arc(184.6503:241.2760:0.2301);
\draw (0.7381,0.3234) arc(222.6749:166.0492:0.2301);
\draw (0.7381,0.3234) arc(123.8579:142.1978:1.2289);
\draw (0.7381,0.3234) arc(288.6011:279.3006:2.4656);
\draw (0.7381,0.3234) arc(193.9508:146.6257:0.2207);
\draw (0.7381,0.3234) arc(231.9754:270.0000:0.2471);
\draw (0.5652,0.1902) arc(241.2760:203.2514:0.7347);
\draw (0.5652,0.1902) arc(193.9508:155.9263:0.5420);
\draw (0.5652,0.1902) arc(250.5765:270.0000:0.8769);
\draw (0.5652,0.1902) arc(155.9263:193.9508:0.7347);
\draw (0.5652,0.1902) arc(165.2268:137.3251:0.6469);
\draw (0.5652,0.1902) arc(146.6257:127.2022:0.8769);
\draw (0.5652,0.1902) arc(203.2514:241.2760:0.5420);
\draw (0.5652,0.1902) arc(117.9017:113.0290:3.3425);
\draw (0.2433,0.5445) arc(250.5765:288.6011:0.5420);
\draw (0.2433,0.5445) arc(297.9017:317.3251:0.8769);
\draw (0.2433,0.5445) arc(288.6011:250.5765:0.7347);
\draw (0.2433,0.5445) arc(187.9946:218.2470:1.0174);
\draw (0.2433,0.5445) arc(165.2268:160.3542:3.3425);
\draw (0.5964,0.5420) arc(241.2760:184.6503:0.2301);
\draw (0.5964,0.5420) arc(302.0683:283.7285:1.2289);
\draw (0.5964,0.5420) arc(137.3251:146.6257:2.4656);
\draw (0.5964,0.5420) arc(231.9754:279.3006:0.2207);
\draw (0.5964,0.5420) arc(203.2514:259.8771:0.2301);
\draw (0.5964,0.5420) arc(193.9508:155.9263:0.2471);
\draw (0.8569,0.1403) arc(155.9263:218.2470:0.1697);
\draw (0.8569,0.1403) arc(193.9508:137.3251:0.1422);
\draw (0.8569,0.1403) arc(137.3251:193.9508:0.2301);
\draw (0.8569,0.1403) arc(222.6749:160.3542:0.1697);
\draw (0.4776,0.7251) arc(288.6011:231.9754:0.2301);
\draw (0.4776,0.7251) arc(203.2514:265.5721:0.1697);
\draw (0.4776,0.7251) arc(270.0000:207.6793:0.1697);
\draw (0.4776,0.7251) arc(231.9754:288.6011:0.1422);
\draw (-0.0484,0.5944) arc(279.3006:317.3251:0.5420);
\draw (-0.0484,0.5944) arc(213.3743:241.2760:1.1240);
\draw (0.5230,-0.2866) arc(212.5520:184.6503:1.1240);
\draw (0.5230,-0.2866) arc(146.6257:108.6011:0.5420);
\draw (0.7151,0.4638) arc(241.2760:180.2224:0.1823);
\draw (0.7151,0.4638) arc(184.6503:245.7039:0.1823);
\node[gnode] at (0.0000,0.0000) {};
\node[gnode] at (0.5964,0.0000) {};
\node[gnode] at (0.4043,0.4385) {};
\node[gnode] at (0.7381,0.3234) {};
\node[gnode] at (0.5652,0.1902) {};
\node[gnode] at (0.2433,0.5445) {};
\node[gnode] at (0.5964,0.5420) {};
\node[gnode] at (0.8569,0.1403) {};
\node[gnode] at (0.4776,0.7251) {};

\node[rnode] at (0.4043,-0.4385) {};
\node[rnode] at (-0.0484,0.5944) {};
\node[rnode] at (0.5230,-0.2866) {};
\node[rnode] at (-0.2354,0.5479) {};
\node[rnode] at (0.2356,0.3896) {};
\node[rnode] at (0.8785,-0.0340) {};
\node[rnode] at (0.6840,0.5349) {};
\node[rnode] at (0.7151,0.4638) {};
\node[rnode] at (0.2625,0.7619) {};
\node[rnode] at (0.6205,0.6229) {};
\node[rnode] at (0.4518,0.0562) {};
\node[rnode] at (0.3501,0.2270) {};
\node[rnode] at (0.7680,0.4980) {};
\node[rnode] at (0.8903,0.2710) {};
\node[rnode] at (0.7673,0.4063) {};
\node[rnode] at (0.8027,-0.0711) {};
\node[rnode] at (0.8218,0.3124) {};
\node[rnode] at (0.3273,0.8160) {};
\node[rnode] at (0.6106,0.7023) {};

\end{tikzpicture}
\caption{Distance $d$ embedding of $G_{28}$}
\label{fig28}
\end{figure}

Since the chromatic number of $G_{28}$ is still $4$ we continue enlarging the graph by adding new vertices. In this second stage, we identify all points $M\in \mathbb{H}^2$ which are not already among the vertices of $G_{28}$ and which are at distance $d$ from at least \emph{three} vertices of $G_{28}$. This leads to a list of $14$ additional points, labeled $A_{29}$ through $A_{42}$.
Table \ref{table2} below contains to coordinates of these points and the neighbors of each such vertex among the vertices of $G_{28}$.

\begin{table}[hbtp]
\begin{centering}
\begin{tabular}{c|c|c}
   New Vertex        & Coordinates           & Neighbors in $G_{28}$       \\
    \hline
$A_{29}$ & $[-16, -8, 12, 3, 16, 8, -12, -2]/2$ & $A_2, A_{10}, A_{25}$\\
$A_{30}$ & $[-4, 2, 4, -1, 0, -4, 2, 0]$ & $A_2, A_{12}, A_{15}$\\
$A_{31}$ & $[112, 82, -48, -17, -48, -8, 26, 10]/19$ & $A_3, A_{17}, A_{23}, A_{25}$\\
$A_{32}$ & $[40, 108, 10, -21, -48, -160, 64, 48]/38$ & $A_4, A_{12}, A_{21}$\\
$A_{33}$ & $[8, 116, 60, 19, 16, -96, 120, 38]/82$ & $A_4, A_{17}, A_{22}, A_{26}$\\
$A_{34}$ & $[8, 12, -6, -1, 16, -16, 0, 4]/2$ & $A_4, A_{23}, A_{24}$\\
$A_{35}$ & $[116, 70, -47, -21, -48, 68, 26, -9]/19$ & $A_5, A_{17}, A_{18}, A_{28}$\\
$A_{36}$ & $[0, 2, -1, 0, 16, 0, -6, 1]$ & $A_6, A_{11},A_{27}$
\end{tabular}
\end{centering}
\end{table}

\begin{table}[hbtp]
\begin{centering}
\begin{tabular}{c|c|c}
$A_{37}$ & $[-12, -8, 8, 2, 0, 4, 0, 0]$ & $A_6, A_{13}, A_{18}$\\
$A_{38}$ & $[56, 60, -24, -18, 128, 72, -44, -14]/19$ & $A_7, A_{11}, A_{21}$\\
$A_{39}$ & $[8, 12, -4, -3, 16, 0, 0, -2]/2$ & $A_7, A_{16}, A_{28}$\\
$A_{40}$ & $[224, 132, -42, -1, 208, 392, -80, -80]/82$ & $A_7, A_{17}, A_{19}, A_{22}$\\
$A_{41}$ & $[16, 6, -11, 1, 0, 0, 2, -1]$ & $A_{15}, A_{21}, A_{23}$\\
$A_{42}$ & $[16, 8, -12, 1, -16, -8, 20, -2/2$, & $A_{21}, A_{27}, A_{28}$
\end{tabular}
\caption{\small{Points in $\mathbb{H}^2$ with at least three neighbors among $A_1,A_2,\ldots, A_{28}$}}
\label{table2}
\end{centering}
\end{table}
As in the prior stage, there are several pairs $(i,j)$, with $29\le i< j\le 42$, for which $f(A_i,A_j)=f$.
These accidental edges are $\{30, 32\}$, $\{33, 34\}$, $\{36, 38\}$, and $\{39, 40\}$.
Eventually, we obtain a graph $G_{42}$ of order $42$ and size $111$ whose distance $d$ embedding is shown if figure \ref{fig42}.
Vertices $A_1$ through $A_{28}$ are shown in yellow, while the new vertices, $A_{29}$ through $A_{42}$, appear in red.

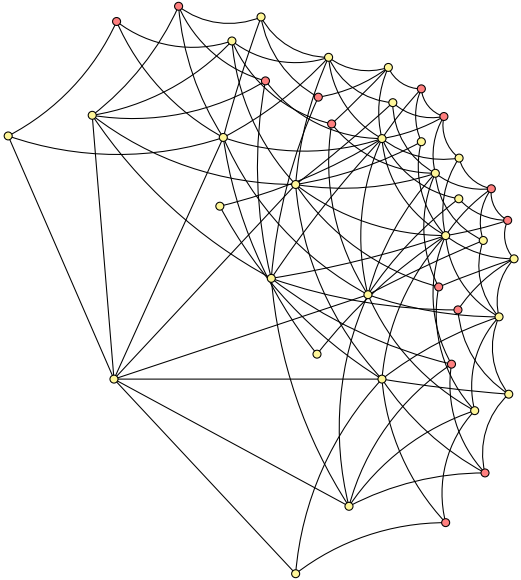
\begin{figure}[ht]
\centering
\begin{tikzpicture}[scale=5.9,line cap=round,line join=round,>=triangle 45]
\tikzstyle{gnode}=[draw=black, fill=yellow!50!white, shape=circle, minimum size=3pt, inner sep=0pt, outer sep=0pt];
\tikzstyle{rnode}=[draw=black, fill=red!50!white, shape=circle, minimum size=3pt, inner sep=0pt, outer sep=0pt];

\draw (0.0000,0.0000) -- (0.5964,0.0000);
\draw (0.0000,0.0000) -- (0.4043,0.4385);
\draw (0.0000,0.0000) -- (0.5652,0.1902);
\draw (0.0000,0.0000) -- (0.2433,0.5445);
\draw (0.0000,0.0000) -- (0.4043,-0.4385);
\draw (0.0000,0.0000) -- (-0.0484,0.5944);
\draw (0.0000,0.0000) -- (0.5230,-0.2866);
\draw (0.0000,0.0000) -- (-0.2354,0.5479);
\draw (0.5964,0.0000) arc(222.6749:184.6503:0.7347);
\draw (0.5964,0.0000) arc(175.3497:137.3251:0.5420);
\draw (0.5964,0.0000) arc(128.0246:108.6011:0.8769);
\draw (0.5964,0.0000) arc(137.3251:175.3497:0.7347);
\draw (0.5964,0.0000) arc(237.9317:207.6793:1.0174);
\draw (0.5964,0.0000) arc(260.6994:265.5721:3.3425);
\draw (0.5964,0.0000) arc(184.6503:222.6749:0.5420);
\draw (0.5964,0.0000) arc(213.3743:241.2760:0.6469);
\draw (0.4043,0.4385) arc(231.9754:270.0000:0.5420);
\draw (0.4043,0.4385) arc(175.3497:155.9263:0.8769);
\draw (0.4043,0.4385) arc(270.0000:231.9754:0.7347);
\draw (0.4043,0.4385) arc(279.3006:298.7240:0.8769);
\draw (0.4043,0.4385) arc(260.6994:288.6011:0.6469);
\draw (0.4043,0.4385) arc(222.6749:184.6503:0.5420);
\draw (0.4043,0.4385) arc(308.0246:312.8972:3.3425);
\draw (0.4043,0.4385) arc(213.3743:254.7432:0.5566);
\draw (0.7381,0.3234) arc(241.2760:184.6503:0.2746);
\draw (0.7381,0.3234) arc(184.6503:241.2760:0.2301);
\draw (0.7381,0.3234) arc(222.6749:166.0492:0.2301);
\draw (0.7381,0.3234) arc(123.8579:142.1978:1.2289);
\draw (0.7381,0.3234) arc(288.6011:279.3006:2.4656);
\draw (0.7381,0.3234) arc(193.9508:146.6257:0.2207);
\draw (0.7381,0.3234) arc(231.9754:270.0000:0.2471);
\draw (0.7381,0.3234) arc(155.9263:209.2076:0.3234);
\draw (0.7381,0.3234) arc(146.6257:125.5135:0.3998);
\draw (0.7381,0.3234) arc(279.3006:288.6011:0.8769);
\draw (0.5652,0.1902) arc(241.2760:203.2514:0.7347);
\draw (0.5652,0.1902) arc(193.9508:155.9263:0.5420);
\draw (0.5652,0.1902) arc(250.5765:270.0000:0.8769);
\draw (0.5652,0.1902) arc(155.9263:193.9508:0.7347);
\draw (0.5652,0.1902) arc(165.2268:137.3251:0.6469);
\draw (0.5652,0.1902) arc(146.6257:127.2022:0.8769);
\draw (0.5652,0.1902) arc(203.2514:241.2760:0.5420);
\draw (0.5652,0.1902) arc(117.9017:113.0290:3.3425);
\draw (0.5652,0.1902) arc(212.5520:171.1831:0.5566);
\draw (0.2433,0.5445) arc(250.5765:288.6011:0.5420);
\draw (0.2433,0.5445) arc(297.9017:317.3251:0.8769);
\draw (0.2433,0.5445) arc(288.6011:250.5765:0.7347);
\draw (0.2433,0.5445) arc(187.9946:218.2470:1.0174);
\draw (0.2433,0.5445) arc(165.2268:160.3542:3.3425);
\draw (0.2433,0.5445) arc(212.5520:184.6503:0.6469);
\draw (0.2433,0.5445) arc(241.2760:203.2514:0.5420);
\draw (0.5964,0.5420) arc(241.2760:184.6503:0.2301);
\draw (0.5964,0.5420) arc(302.0683:283.7285:1.2289);
\draw (0.5964,0.5420) arc(137.3251:146.6257:2.4656);
\draw (0.5964,0.5420) arc(231.9754:279.3006:0.2207);
\draw (0.5964,0.5420) arc(203.2514:259.8771:0.2301);
\draw (0.5964,0.5420) arc(193.9508:155.9263:0.2471);
\draw (0.5964,0.5420) arc(270.0000:216.7186:0.3234);
\draw (0.5964,0.5420) arc(146.6257:137.3251:0.8769);
\draw (0.5964,0.5420) arc(279.3006:300.4128:0.3998);
\draw (0.8569,0.1403) arc(155.9263:218.2470:0.1697);
\draw (0.8569,0.1403) arc(193.9508:137.3251:0.1422);
\draw (0.8569,0.1403) arc(137.3251:193.9508:0.2301);
\draw (0.8569,0.1403) arc(222.6749:160.3542:0.1697);
\draw (0.4776,0.7251) arc(288.6011:231.9754:0.2301);
\draw (0.4776,0.7251) arc(203.2514:265.5721:0.1697);
\draw (0.4776,0.7251) arc(270.0000:207.6793:0.1697);
\draw (0.4776,0.7251) arc(231.9754:288.6011:0.1422);
\draw (0.4043,-0.4385) arc(128.0246:90.0000:0.5420);
\draw (-0.0484,0.5944) arc(279.3006:317.3251:0.5420);
\draw (-0.0484,0.5944) arc(213.3743:241.2760:1.1240);
\draw (-0.0484,0.5944) arc(308.0246:335.9263:0.6469);
\draw (-0.0484,0.5944) arc(260.6994:302.0683:0.5566);
\draw (0.5230,-0.2866) arc(212.5520:184.6503:1.1240);
\draw (0.5230,-0.2866) arc(146.6257:108.6011:0.5420);
\draw (0.5230,-0.2866) arc(117.9017:90.0000:0.6469);
\draw (0.5230,-0.2866) arc(165.2268:123.8579:0.5566);
\draw (-0.2354,0.5479) arc(297.9017:335.9263:0.5420);
\draw (0.8785,-0.0340) arc(132.8972:193.9508:0.1823);
\draw (0.8785,-0.0340) arc(236.8481:184.6503:0.2512);
\draw (0.6840,0.5349) arc(204.0737:155.9263:0.1460);
\draw (0.7151,0.4638) arc(241.2760:180.2224:0.1823);
\draw (0.7151,0.4638) arc(184.6503:245.7039:0.1823);
\draw (0.7151,0.4638) arc(155.9263:207.4181:0.2949);
\draw (0.7151,0.4638) arc(231.9754:276.7895:0.1697);
\draw (0.7151,0.4638) arc(270.0000:218.5082:0.2949);
\draw (0.7151,0.4638) arc(193.9508:149.1368:0.1697);
\draw (0.2625,0.7619) arc(203.2514:256.5328:0.3234);
\draw (0.2625,0.7619) arc(288.6011:231.9754:0.2746);
\draw (0.6205,0.6229) arc(227.5475:281.8117:0.1292);
\draw (0.3501,0.2270) arc(231.9754:256.5328:1.0466);
\draw (0.3501,0.2270) arc(193.9508:169.3935:1.0466);
\draw (0.3501,0.2270) arc(250.5765:270.0000:1.2492);
\draw (0.3501,0.2270) arc(175.3497:155.9263:1.2492);
\draw (0.7680,0.4980) arc(193.9508:258.1883:0.0937);
\draw (0.7680,0.4980) arc(231.9754:167.7379:0.0937);
\draw (0.8903,0.2710) arc(108.6011:112.7678:2.4656);
\draw (0.8903,0.2710) arc(222.6749:155.9263:0.0799);
\draw (0.8903,0.2710) arc(118.7240:146.6257:0.3520);
\draw (0.7673,0.4063) arc(221.8525:270.0000:0.1460);
\draw (0.8027,-0.0711) arc(137.3251:193.9508:0.2746);
\draw (0.8027,-0.0711) arc(222.6749:169.3935:0.3234);
\draw (0.8218,0.3124) arc(198.3787:144.1146:0.1292);
\draw (0.3273,0.8160) arc(293.0290:231.9754:0.1823);
\draw (0.3273,0.8160) arc(189.0782:241.2760:0.2512);
\draw (0.6106,0.7023) arc(317.3251:313.1585:2.4656);
\draw (0.6106,0.7023) arc(203.2514:270.0000:0.0799);
\draw (0.6106,0.7023) arc(307.2022:279.3006:0.3520);
\draw (0.8257,-0.2114) arc(222.6749:171.1831:0.2949);
\draw (0.8397,0.4289) arc(172.8386:241.2760:0.0711);
\draw (0.1438,0.8401) arc(203.2514:254.7432:0.2949);
\draw (0.6840,0.6540) arc(184.6503:253.0876:0.0711);
\node[gnode] at (0.0000,0.0000) {};
\node[gnode] at (0.5964,0.0000) {};
\node[gnode] at (0.4043,0.4385) {};
\node[gnode] at (0.7381,0.3234) {};
\node[gnode] at (0.5652,0.1902) {};
\node[gnode] at (0.2433,0.5445) {};
\node[gnode] at (0.5964,0.5420) {};
\node[gnode] at (0.8569,0.1403) {};
\node[gnode] at (0.4776,0.7251) {};
\node[gnode] at (0.4043,-0.4385) {};
\node[gnode] at (-0.0484,0.5944) {};
\node[gnode] at (0.5230,-0.2866) {};
\node[gnode] at (-0.2354,0.5479) {};
\node[gnode] at (0.2356,0.3896) {};
\node[gnode] at (0.8785,-0.0340) {};
\node[gnode] at (0.6840,0.5349) {};
\node[gnode] at (0.7151,0.4638) {};
\node[gnode] at (0.2625,0.7619) {};
\node[gnode] at (0.6205,0.6229) {};
\node[gnode] at (0.4518,0.0562) {};
\node[gnode] at (0.3501,0.2270) {};
\node[gnode] at (0.7680,0.4980) {};
\node[gnode] at (0.8903,0.2710) {};
\node[gnode] at (0.7673,0.4063) {};
\node[gnode] at (0.8027,-0.0711) {};
\node[gnode] at (0.8218,0.3124) {};
\node[gnode] at (0.3273,0.8160) {};
\node[gnode] at (0.6106,0.7023) {};

\node[rnode] at (0.7381,-0.3234) {};
\node[rnode] at (0.8257,-0.2114) {};
\node[rnode] at (0.7226,0.2077) {};
\node[rnode] at (0.7511,0.0337) {};
\node[rnode] at (0.8397,0.4289) {};
\node[rnode] at (0.8761,0.3577) {};
\node[rnode] at (0.4844,0.5750) {};
\node[rnode] at (0.1438,0.8401) {};
\node[rnode] at (0.0058,0.8058) {};
\node[rnode] at (0.3371,0.6720) {};
\node[rnode] at (0.6840,0.6540) {};
\node[rnode] at (0.7342,0.5917) {};
\node[rnode] at (0.7655,0.1560) {};
\node[rnode] at (0.4547,0.6353) {};
\end{tikzpicture}
\caption{Distance $d$ embedding of $G_{42}$}
\label{fig42}
\end{figure}
Since the chromatic number of $G_{42}$ is still equal to $4$, we apply to this graph the same procedure we used earlier on $G_{28}$. More precisely, we identify a set consisting of $26$ new points, each of which is at distance $d$ from at least three vertices of $G_{42}$. Adjoining these new vertices to $G_{42}$ we obtain a graph $G_{68}$ with $68$ vertices and $201$ edges whose distance $d$ embedding is shown in figure \ref{fig68}. Applying this augmentation process four more times we obtain the graphs $G_{119}$, $G_{226}$, $G_{455}$, and finally, $G_{762}$, whose orders are indicated by their subscripts and whose sizes are $385$, $786$, $1679$, and $2983$, respectively.

\begin{figure}[htbp!]
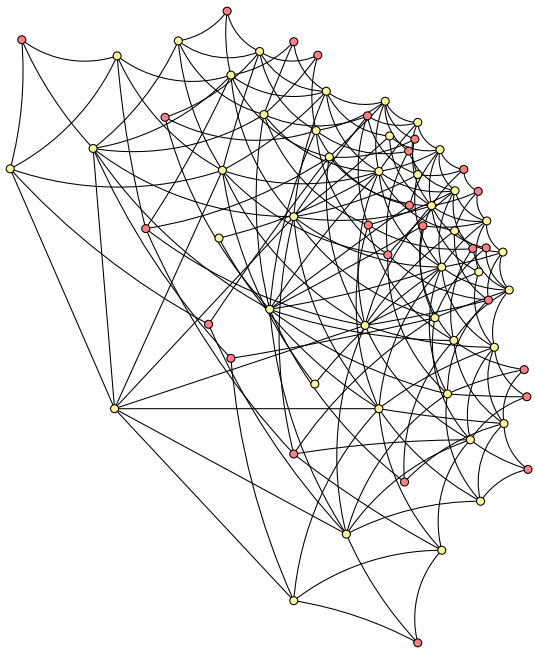

\centering

\caption{Distance $d$ embedding of $G_{68}$}
\label{fig68}
\end{figure}

\begin{figure}[htbp!]
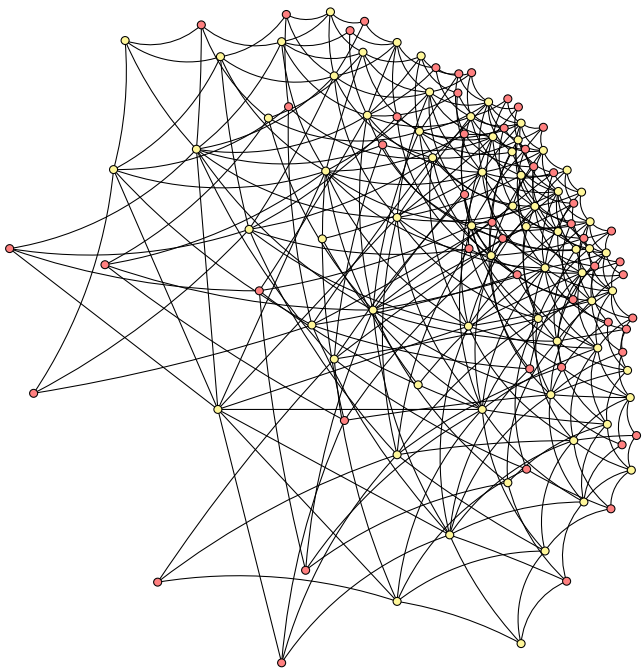

\centering
%
\caption{Distance $d$ embedding of $G_{119}$}
\label{fig119}
\end{figure}

\section{Graph coloring algorithm}

This section outlines a simple algorithm that can be used to verify that $G_{762}$ is not $4$-chromatic.
To begin we fix an ordering of the vertices, the one given by the construction described above works well.
The algorithm maintains two data structures.  The first of these is a simple array of colors containing
the current color of each vertex, where the value UNCOLORED is viewed as a color.
The second data structure maintains the set of colors that could be assigned to each uncolored vertex
without creating a monochromatic edge.
Manipulating the second data structure efficiently is the key to obtaining good performance.
The outline of the algorithm is given below.

When the algorithm is applied to $G_{762}$ with the given vertex order, the depth of the
recursion never exceeds $78$, at which point (or earlier) there is an uncolored vertex with no
feasible colors remaining. The running time on an Intel i7-11700 is approximately 0.60 seconds.
In fact the subgraph induced by the first $622$ vertices is $5$-chromatic, which
takes 0.52 seconds to run and reaches a maximum recursion depth of $108$.  The subgraph
induced by the first $621$ vertices is $4$-chromatic.  A coloring is found in
0.02 seconds.

As it can be expected from the figures describing $G_{68}$ and $G_{119}$, the embedding of $G_{762}$ as a $d$ distance graph is rather messy. Nevertheless, we include a plot of the vertex set, mainly for illustrative purposes .

\begin{figure}[ht]
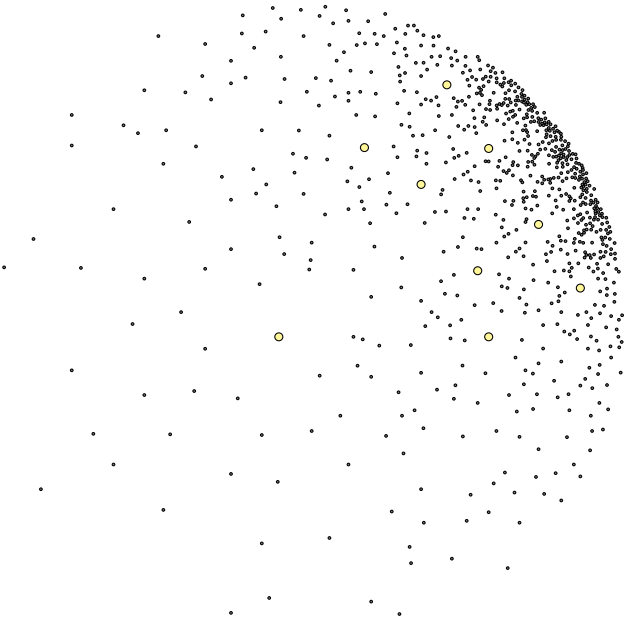

\centering

\caption{Vertex set of $G_{762}$; in yellow, the vertices of the starting graph, $G_9$}
\label{fig762}
\end{figure}

\begin{algorithm}[H]
\begin{small}
\caption{Coloring Search}
\label{coloralgo}
\begin{algorithmic}
\algnewcommand{\BigComment}[1]{\State \(\triangleright\) #1}
\Procedure{Search}{$vertex,colors,feasible$} \Comment{The graph is fixed and global}
  \If{$vertex = n$} \Comment{Vertices numbered $0 \cdots n-1$}
    \State Output $colors$
    \State Return
  \EndIf

  \State
  \BigComment{If the color of a vertex has already been forced, we skip ahead to the next level.}
  \State

  \If{$color[vertex] \not= UNCOLORED$} \Comment{Vertex already colored}
    \State Search($vertex+1,colors, feasible$)
    \State Return
  \EndIf
  \State
  \BigComment{The $Assign$ function colors a vertex and recursively considers all implications.
    This may force colors on other vertices.  Conflicts may be discovered, in which case it returns False.
    The $UnAssign$ function restores the previous state of the coloring.}
  \State
  \For{each color $c$ in $feasible[c]$}
    \If{Assign($vertex, c, colors, feasible$)}
      \State Search($vertex+1,colors,feasible$)
      \State UnAssign($vertex, c, colors, feasible$)
    \EndIf
  \EndFor
\EndProcedure
\end{algorithmic}
\end{small}
\end{algorithm}

\section{Conclusions}
We proved that for the choice $d=\arccosh\left(\frac{9+5\sqrt{5}+(4+\sqrt{5})\sqrt{22-2\sqrt{5}}}{22}\right)\approx 1.37503$,
every $4$-coloring of the hyperbolic plane $\mathbb{H}^2$ contains two identically colored points which are distance $d$ apart.
Combining this with the result \eqref{parlierpetit} of Parlier and Petit  mentioned earlier it follows that
$5\le \chi(\mathbb{H}^2,d)\le 9$.
The full list of coordinates of the embedding of $G_{762}$, as well the incidence list are available at \cite{adjacencylist, coordinates}.

Our construction was dictated by condition \ref{choice} which lead to a relatively small value of $d$. Of course, other choices are possible, and it may very well be the case that for larger values of $d$ one may prove better lower bounds for $\chi(\mathbb{H}^2,d)$. It would also be of interest to find a value of $d$ for which a small (ideally, humanly verifiable) $5$-chromatic distance $d$ graph exists in $\mathbb{H}^2$.

We conclude with a very interesting question of Kahle \cite{kahle}: \emph{ Prove or disprove the existence of an universal upper bound for $\chi(\mathbb{H}^2, d)$ where $d>0$.}
\emph{}

\end{document}